\documentclass[11pt]{article}
\usepackage{}
\usepackage[total={6in, 9in}]{geometry}
 \usepackage{amsfonts}
\usepackage{amsthm}
\usepackage{amssymb}
\usepackage{amsmath,amsfonts}
\usepackage{array}
\usepackage{diagbox}
\usepackage{cancel}
\usepackage{mathrsfs}
\usepackage{cases}
\usepackage{latexsym,bm}
\usepackage{indentfirst}
\usepackage{color}
\usepackage{ifpdf}
\usepackage{graphicx}
\usepackage{psfrag}
\usepackage{enumerate}
\usepackage{dsfont}
\usepackage[
pdfauthor={},
pdftitle={},
pdfstartview=XYZ,
bookmarks=true,
colorlinks=true,
linkcolor=blue,
urlcolor=blue,
citecolor=blue,
bookmarks=true,
linktocpage=true,
hyperindex=true
]{hyperref}

\usepackage[natural]{xcolor}
\usepackage{tikz}
\usepackage{subfigure,tikz}
\usepackage{tikz-3dplot}
\usetikzlibrary{decorations.pathreplacing}
\usepackage{pgf,tikz,pgfplots}
\usepackage{mathrsfs}

\newtheorem{THM}{\textbf{Theorem}}

\newtheorem{LEM}{\textbf{Lemma}}
\newtheorem{CLA}{\textbf{Claim}}[section]
\newtheorem{CON}[THM]{\textbf{Conjecture}}

\newcommand{\pf}{\noindent\textbf{Proof}.\quad}

\begin{document}
\title{Graceful Labeling of Two Families of Spiders}
\author{Songling Shan\footnote{Auburn University, Department of Mathematics and Statistics, Auburn, AL 36849.
		Email: {\tt szs0398@auburn.edu}.   
		Supported in part by NSF grant DMS-2451895.}
		\qquad 
		 Yucheng Zhong\footnote{Auburn University, Department of Mathematics and Statistics, Auburn, AL 36849.
		 	Email:	{\tt yzz0237@auburn.edu}. }
}

\date{\today}
\maketitle

\begin{abstract}

 A \emph{graceful labeling} of a graph  $G$ is an injective function $f : V(G) \to \{0, \ldots, |E(G)|\}$ such that $\{\,|f(u)-f(v)| : uv \in E(G)\,\} = \{1, \ldots, |E(G)|\}$. If such a labeling exists,  then we call $G$  \emph{graceful}. Introduced by Rosa  in 1967, graceful labeling has been widely studied, and the Graceful Tree Conjecture asserts that every tree is graceful. The conjecture  is known to hold for several classes of trees, including caterpillars, trees with at most four leaves, trees of diameter at most five, and certain spiders. An important subclass  is that of \emph{$\alpha$-labelings}, where a graceful labeling $f$ admits an integer $\alpha$ such that each edge joins a vertex with label at most $\alpha$ to one with label greater than $\alpha$. A result from 1982 by Huang, Kotzig, and Rosa shows that if $H$ has an $\alpha$-labeling with a vertex $u$ labeled $0$ or $\alpha$, and $G$ has a graceful labeling with a vertex  $v$ labeled $0$, then identifying $u$ and $v$  yields a graceful graph, though this requires a $0$-labeled vertex in $G$. We prove a related result that relaxes this condition: if $G$ has a graceful labeling $f$ such that  $f(u)+\lfloor n/2 \rfloor + 1 \le n$ and $n \not\equiv 1 \pmod{4}$, where $u\in V(G)$ and $n\ge 2$ is an integer, then joining $u$ to an end vertex of the vertex-disjoint $n$-vertex path  $P_n$ yields a graceful graph. As an application, we show that any spider with legs $L_1,\ldots,L_s$ ($s \ge 1$) satisfying $|E(L_{2})| \ge 2|E(L_1)|+ 4$ and $|E(L_{i+1})| \ge 2|E(L_i)|+ 2$ for $i \in \{2,\ldots, s-1\}$ is graceful. Furthermore,  we give an explicit graceful  labeling for spiders with one leg of arbitrary length and all others of length at most two such that the center is labeled by $0$.  This labeling enables  the construction of larger graceful spiders by attaching paths at the center.
\medskip 


\end{abstract}

\section{Introduction}

For two integers $p$ and $q$, we let $[p,q]=\{i\in \mathbb{Z}: p \le i\le q\}$. 
 A  \emph{graceful labeling} of  a graph  $G$ is an injective function $f : V(G) \to [0,|E(G)|]$ such that $\{\,|f(u)-f(v)| : uv \in E(G)\,\} = [1,|E(G)|]$.  
If $G$ admits a graceful labeling, then $G$ is \emph{graceful}.  

Graceful labelings were first introduced by Rosa~\cite{Rosa1967} under the name of $\beta$-valuations, and the term “graceful labeling” was due to Golomb~\cite{Golomb72}. The major open problem on graceful labelings is the Graceful Tree Conjecture, proposed by Rosa in 1967~\cite{Rosa1967}. It is also known as the Kotzig–Ringel–Rosa Conjecture, reflecting its close connection to the Ringel Conjecture~\cite{Ringel1964} and Kotzig’s earlier work on related graph decomposition and labeling problems, including special cases of the conjecture.

\begin{CON}[Graceful Tree Conjecture]\label{con:GTC}
	Every tree is graceful. 
\end{CON}

 While Conjecture~\ref{con:GTC} has attracted considerable attention, it has been verified only for certain special classes of trees, including caterpillars~\cite{Rosa1967} (a caterpillar is a tree with the property that the removal of all leaves leaves a path), trees with at most four end-vertices~\cite{huang}, and trees with diameter at most five~\cite{HrncicarHaviar2001}. For additional classes of trees known to admit graceful labelings, we refer the reader to Gallian’s dynamic survey~\cite{GallianSurvey}. In this paper, our focus is on graceful spiders.

 A \emph{spider} is a tree with at most one vertex (called the  \emph{center}) of degree greater than two. Let $S$ be a spider such that the degree of the center is at least three (otherwise, the spider is simply a path).  Each  path in $S$ connecting the center and 
 a leaf  is called a \emph{leg} of $S$.     The general question whether all spiders are graceful is still open, but some classes of spiders have been shown to be graceful. 
Bahls, Lake, and Wertheim~\cite{Bahls} proved that spiders whose leg lengths differ by at most one are graceful. Panpa and Poomsa-ard~\cite{PanpaPoomsaard2016} showed that all spiders  with at most four legs of length greater than one admit a graceful labeling, and this result was extended by Panpa, Imnang, and Wasuanankul~\cite{PanpaImnangWasuanankul2025} to spiders with at most five such legs. For additional classes of spiders known to admit graceful labelings, we refer the reader to Gallian’s dynamic survey~\cite{GallianSurvey}.

A special type of graceful labeling, called an \emph{$\alpha$-labeling}, has been used as a constructive tool; see, for example, Panpa and Poomsa-ard~\cite{PanpaPoomsaard2016}. A graceful labeling $f$ is said to be an $\alpha$-labeling if there exists an integer $\alpha$ such that $f(u) \leq \alpha < f(v)$ for each edge $uv$ with $f(u) < f(v)$. The integer $\alpha$ is called the \emph{index} of $f$. Clearly, a graph must be bipartite in order to admit an $\alpha$-labeling. The following lemma of Huang, Kotzig, and Rosa~\cite{huang} provides a useful method for constructing larger graceful graphs from smaller ones.

\begin{LEM}[\cite{huang}]\label{lem:vertex amalgamation}
	Let $G$ and $H$ be vertex-disjoint graphs with $u \in V(G)$ and $v \in V(H)$.  If $G$ admits an $\alpha$-labeling with index $\alpha$ such that $u$ is labeled by $\alpha$ or $0$, and $H$ admits a graceful labeling such that $v$ is labeled by $0$, then the graph   obtained  from $G$ and $H$ by identifying $u$ and $v$ is graceful.
\end{LEM}

The graceful labeling in Lemma~\ref{lem:vertex amalgamation} is obtained by shifting the vertex labels of \(H\) upward by \(\alpha\), and, when \(u\) is labeled \(\alpha\), shifting the labels in the upper part of the \(\alpha\)-labeling of \(G\) upward by \(|E(H)|\). If instead \(u\) is labeled \(0\), one first applies the reverse \(\alpha\)-labeling of \(G\), thereby obtaining another \(\alpha\)-labeling in which \(u\) receives the label \(\alpha\). The same construction then applies. In either case, the vertex formed by identifying \(u\) and \(v\) is assigned the label \(\alpha\).

For $n \geq 1$, let $P_n$ denote a path on $n$ vertices. Cattell~\cite{Cattell} proved that every path $P_n$ admits an $\alpha$-labeling in which a prescribed vertex can be assigned any label in $[0,n-1]$, provided that $n \not\equiv 1 \pmod{4}$. Consequently, if a spider admits a graceful labeling in which its center is labeled by $0$, then one can construct a larger graceful spider by attaching a suitable path at its center via Lemma~\ref{lem:vertex amalgamation}. However, after this operation, the center of the resulting spider is no longer labeled by $0$;
it is labeled by $\alpha$. 

Our first contribution addresses this issue by allowing iterative attachment of paths to the center of a graceful spider, provided the path length is relatively large. Let $G$ and $H$ be vertex-disjoint graphs with $u \in V(G)$ and $v \in V(H)$. We define $G(u)\oplus H(v)$ to be the graph obtained from the disjoint union of $G$ and $H$ by adding the edge $uv$. We now state our result.

\begin{THM}\label{thm:attching lemma}
	Let $G$ be a graph with $u \in V(G)$, and let $P_n$ be a path vertex-disjoint from $G$ with $v$ an endvertex of $P_n$. Suppose that $G$ admits a graceful labeling $f$ such that
	\[
	f(u) + \left\lfloor \frac{n}{2} \right\rfloor + 1 \leq n,
	\]
	and that $n \not\equiv 1 \pmod{4}$. Then $G(u)\oplus P_n(v)$ admits a graceful labeling $h$ such that
	\[
	h(w) = f(w) + \left\lfloor \frac{n}{2} \right\rfloor \quad \text{for all } w \in V(G).
	\]
\end{THM}

As an application of Theorem~\ref{thm:attching lemma}, we obtain the following result.

 \begin{THM}\label{main1}
 	Let $S$ be a spider with legs $L_1,L_2,...L_s$,  and let $\ell_i=|E(L_i)|$ for each $i\in [1,s]$. 
 	Suppose that    $\ell_{i+1}\geq 2\ell_i+2$ for any $i\in [2,s-1]$, and $\ell_{2}\geq 2\ell_1+2$ if $\ell_2 \not\equiv 1 \pmod 4$
 	and $\ell_{2}\geq 2\ell_1+4$  otherwise. 
 	Then $S$ is graceful.
 \end{THM}

Saengsura and Poomsa-Ard~\cite{saengsura} proved a  result  similar as  Theorem~\ref{main1} where  they required $\ell_{i+1}\geq 2\ell_i+2$ but $\ell_i$ is even for all $i\in [3,s]$.

Furthermore, we give an explicit graceful  labeling for spiders with one leg of arbitrary length and all others of length at most two such that the center is labeled by $0$.

\begin{THM}\label{thm:graceful-construction}
Let $S$ be a spider with one leg of arbitrary length and all others of length at most two. 
Then $S$ has a graceful labeling such that its center is labeled by zero. 
\end{THM}

As an application of Lemma~\ref{lem:vertex amalgamation} and Theorem~\ref{thm:graceful-construction}, we have the following result.

\begin{THM}\label{main3}
	If $S$ is a spider with all but three legs of length at most $2$, then  $S$ is graceful.
\end{THM}

The remainder of this paper is organized as follows. In the next section, 
we prove Theorems~\ref{thm:attching lemma} and~\ref{main1}.  The proofs  of
Theorems~\ref{thm:graceful-construction} and~\ref{main3} are given in the last section.

\section{Proof of Theorems~\ref{thm:attching lemma} and~\ref{main1}}

Let $G$ be a graph and $f:V(G) \to [0,|E(G)|]$.   For $uv\in E(G)$, define $f(uv) =|f(u)-f(v)|$. 
For $X\subseteq V(G)$ and $F\subseteq E(G)$, 
we let $f(X)=\{f(x): x\in X\}$ and $f(F) =\{f(e): e\in F\}$. 
For an integer $\alpha$ and a set $[p,q]$ with $p,q\in \mathbb{Z}$, we let $\alpha+[p,q]=\{\alpha+i: i\in [p,q]\}$. 
We will need the following result on graceful labelings of paths.

\begin{LEM}\label{lem:alpha-labeling-Pn}
Let   $P_n$ be a path and $v\in V(P_n)$. 
\begin{enumerate}[(a)]
	\item  Then $P_n$ has a graceful labeling such that $v$ is labelled by $0$.  (See Cattell~\cite[Theorem 3]{Cattell}.)
		\item Then  $P_n$ has   an $\alpha$-labeling $f$   such that $f(v) = 0$ except 
	$n=5$ and    $v$ is the central vertex. (See Rosa~\cite{rosa1977labeling} and Cattell~\cite[Corollary 1]{Cattell}.)  
	\item  If $v$ is an endvertex, then $P_n$ has an $\alpha$-labeling $f$  such that $f(v)=i$ for any $i\in [0,n-1]$ unless $n=4s+1$ for some $s\in \mathbb{N}$ and $i \in \{s, 3s\}$. (See Cattell~\cite[Corollary 1]{Cattell}.)

\end{enumerate}
\end{LEM}

 \begin{proof}[Proof of Theorem~\ref{thm:attching lemma}]
    Since $n\not\equiv 1 \pmod 4$ and $f(u)+\lfloor\frac{n}{2}\rfloor \le n-1 =|E(P_n)|$,  by Lemma~\ref{lem:alpha-labeling-Pn}(c), $P_n$ has 
 	an $\alpha$-labeling $g$ with  the index 
 	\begin{align*}
 	\alpha &=\lfloor\frac{n}{2} \rfloor -1,  \quad \text{and} \\ 
 	g(v) &=f(u)+\lfloor\frac{n}{2}\rfloor. 
 	\end{align*}
 	
 	Let $m=|E(G)|$ and  $G'=G(u)\oplus P_n(v)$.  Note that $E(G')=E(G) \cup E(P_n) \cup \{uv\}$. 
 	 We now define a labeling $h$ of $G'$ as follows: 
 \begin{align*}
 	h(w) &= f(w)+\lfloor \frac{n}{2} \rfloor, &&\text{for all } w\in V(G);\\
 	h(x) &= g(x), &&\text{for all } x \in V(P_{n})\text{ such that } 0\le g(x)\le \lfloor\frac{n}{2}\rfloor-1;\\
 h(x) &= g(x)+m+1, &&\text{for all } x \in V(P_{n})\text{ such that } \lfloor\frac{n}{2}\rfloor\le g(x)\le n-1.
 \end{align*}
 Figure~\ref{fig:attaching_lemma_example} illustrates the construction of $h$. 
 
 \begin{figure}[htpb] 
 	\centering 
 		\begin{tikzpicture}
 			\draw [black, thick] (-8,-1)--(-9.5,-2);
 			\draw [black, thick] (-8,-1)--(-8.75,-2)--(-8.75,-3);
 			\draw [black, thick] (-8,-1)--(-8,-2)--(-8,-3)--(-8,-4);
 			\node[circle, draw, fill=white, scale=.7] (1) at (-8,-1) {0};
 			\node[circle, draw, fill=white, scale=.7] (2) at (-9.5,-2) {6};
 			\node[circle, draw, fill=white, scale=.7] (3) at (-8.75,-2) {5};
 			\node[circle, draw, fill=white, scale=.7] (4) at (-8.75,-3) {1};
 			\node[circle, draw, fill=white, scale=.7] (5) at (-8,-2) {3};
 			\node[circle, draw, fill=white, scale=.7] (6) at (-8,-3) {4};
 			\node[circle, draw, fill=white, scale=.7] (7) at (-8,-4) {2};
 			\node [] at (-9,-5) {$G$ with $f(u)=3$};
 			\node [] at (-7.7,-1.6) {$u$};
 			\node [] at (-7.2,-1.6) {$v$};
 			
 			\draw[black,thick] (-7.05,-2)--(-6.35,-2)--(-5.65,-2)--(-4.95,-2)--(-4.95,-2.7)--(-4.95,-3.4)--(-4.95,-4.1);
 			\node[circle, draw, fill=white, scale=.7] (8) at (-7.05,-2) {6};
 			\node[circle, draw, fill=white, scale=.7] (9) at (-6.35,-2) {0};
 			\node[circle, draw, fill=white, scale=.7] (10) at (-5.65,-2) {5};
 			\node[circle, draw, fill=white, scale=.7] (11) at (-4.95,-2) {1};
 			\node[circle, draw, fill=white, scale=.7] (12) at (-4.95,-2.7) {4};
 			\node[circle, draw, fill=white, scale=.7] (13) at (-4.95,-3.4) {2};
 			\node[circle, draw, fill=white, scale=.7] (14) at (-4.95,-4.1) {3};
 			\node[] at(-3.85,-5) {$P_7$ with an $\alpha$-labeling};
 			\draw [black, thick] (0,-1)--(-1.5,-2);
 			\draw [black, thick] (0,-1)--(-.75,-2)--(-.75,-3);
 			\draw [black, thick] (0,-1)--(0,-2)--(0,-3)--(0,-4);
 			\draw[black,thick] (0.95,-2)--(1.65,-2)--(2.35,-2)--(3.05,-2)--(3.05,-2.7)--(3.05,-3.4)--(3.05,-4.1);
 			\draw[red,ultra thick] (0,-2)--(0.95,-2); 
 			
 			\node[circle, draw, fill=white, scale=.7] (15) at (0,-1) {3};
 			\node[circle, draw, fill=white, scale=.7] (16) at (-1.5,-2) {9};
 			\node[circle, draw, fill=white, scale=.7] (17) at (-.75,-2) {8};
 			\node[circle, draw, fill=white, scale=.7] (18) at (-.75,-3) {4};
 			\node[circle, draw, fill=white, scale=.7] (19) at (0,-2) {6};
 			\node[circle, draw, fill=white, scale=.7] (20) at (0,-3) {7};
 			\node[circle, draw, fill=white, scale=.7] (21) at (0,-4) {5};
 			\node[] at(0.3,-1.6) {$u$};
 			
 			\node[circle, draw, fill=white, scale=.65] (22) at (0.95,-2) {13};
 			\node[circle, draw, fill=white, scale=.7] (23) at (1.65,-2) {0};
 			\node[circle, draw, fill=white, scale=.65] (24) at (2.35,-2) {12};
 			\node[circle, draw, fill=white, scale=.7] (25) at (3.05,-2) {1};
 			\node[circle, draw, fill=white, scale=.65] (26) at (3.05,-2.7) {11};
 			\node[circle, draw, fill=white, scale=.7] (27) at (3.05,-3.4) {2};
 			\node[circle, draw, fill=white, scale=.65] (28) at (3.05,-4.1) {10};
 			\node[] at(0.8,-1.6) {$v$};
 			
 			\node [] at (1,-5) {Labeling $h$ of  $G(u)\oplus P_7(v)$};
 		\end{tikzpicture}
 	\caption{An illustration of the labeling $h$.}
 	\label{fig:attaching_lemma_example} 
 \end{figure}
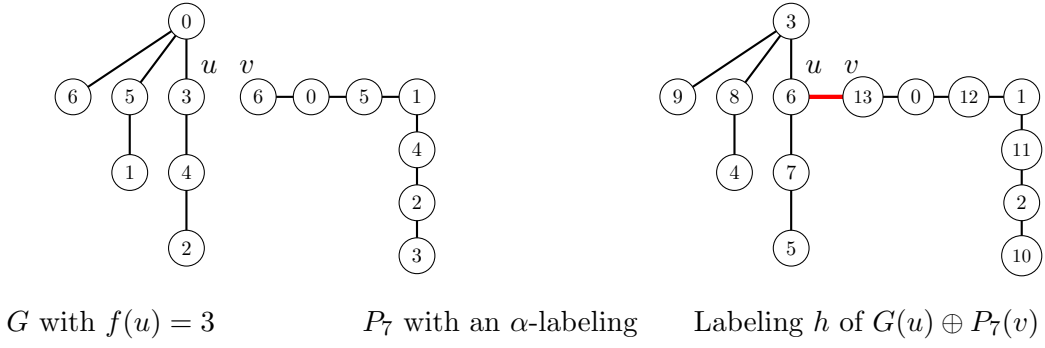

 We first verify that $h$ is injective  from $V(G')$ to $[0,m+n]$. 
Since $f: V(G) \to [0,m]$ and $g: V(P_n) \to [0, n-1]$ are both injective, by the definition of  $h$, we have 
\begin{align*}
h(V(G)) &=\lfloor \frac{n}{2} \rfloor+[0,m]=\left[\lfloor \frac{n}{2} \rfloor,m+\lfloor \frac{n}{2}\rfloor\right], \\
h(V(P_n)) &= [0,\lfloor\frac{n}{2}\rfloor-1] \cup \left(m+1+[\lfloor\frac{n}{2}\rfloor, n-1]\right) \\
 &= [0,\lfloor\frac{n}{2}\rfloor-1] \cup \left[m+\lfloor\frac{n}{2}\rfloor+1, m+n \right]. 
\end{align*}
Thus $h$ is an injective  function from $V(G')$ to $[0,m+n]$. 

Next we verify that $h(E(G'))=[1,m+n]$.  As 
$h(x) = f(x)+\lfloor \frac{n}{2} \rfloor$ for all   $x\in V(G)$ 
and $f(E(G))=[1,m]$, we have  $h(E(G))=[1,m]$. 
As $g$ is an $\alpha$-labeling of  $P_n$, by the definition of $h$, 
we have $h(E(P_n)) =m+1+[1,n-1] =[m+2, m+n]$. 
Furthermore, we have $h(u) =f(u)+\lfloor \frac{n}{2} \rfloor$
and $h(v) =g(v)+m+1=f(u)+\lfloor\frac{n}{2}\rfloor+m+1$ since $g(v)\ge \lfloor\frac{n}{2}\rfloor$, and so $h(uv) =m+1$. 
Therefore, 
$$h(E(G')) =[1,m] \cup \{m+1\} \cup [m+2, m+n] =[1,m+n].$$

The above shows that $h$ is a graceful labeling of $G'$.  The second part of 
the statement is clear by the definition of $h$. 
 \end{proof}

\begin{proof}[Proof of Theorem~\ref{main1}]
Let $S$ be a spider with legs $L_1,L_2,\ldots, L_s$, and let    $\ell_i=|E(L_i)|$ such that   $\ell_{i+1}\geq 2\ell_i+2$ for any $i\in [2,s-1]$, and $\ell_{2}\geq 2\ell_1+2$ if $\ell_2 \not\equiv 1 \pmod 4$
and $\ell_{2}\geq 2\ell_1+4$  otherwise.  
We may assume that $s\ge 3$, as we are done by Lemma~\ref{lem:alpha-labeling-Pn}(a) otherwise.  
If there exists $i\in [2,s]$ such that $\ell_i \equiv 1 \pmod 4$, then we let  $k_1, \ldots, k_r$ be all the indices with 
$k_1<k_2< \ldots <k_r$ such that $\ell_{k_j} \equiv 1 \pmod 4$ for all $j\in [1,r]$. 

We let  $x$ be the center of $S$, and let $xy_i \in E(L_i)$ be the edge of $L_i$ incident with $x$. 
For each $j\in [1,r]$, we further let $y_{k_j}z_{k_j}$ be the edge of $L_{k_j}-x$ that is incident with $y_{k_j}$. 
We let $$S_1=S[V(L_1) \cup \{y_{k_j}: j\in [1,r]\}].$$
By Lemma~\ref{lem:alpha-labeling-Pn}(a), $S[V(L_1 )]$ has a graceful labeling  $f_1'$ such that $x$ is labeled by $0$. 
We now define a  labeling $f_1$ for $S_1$: 
\begin{numcases}{f_1(u)=}
f_1'(u) & \text{if $u\in V(L_1)$}, \nonumber  \\ 
\ell_1+j & \text{if $u=y_{k_j}, \quad  j\in [1,r]$}.  \nonumber 
\end{numcases}
Since $f_1'(x)=0$ and $f_1'$ is a graceful labeling of $S[V(L_1 )]$, it is clear that $f_1$ is a graceful labeling of $S_1$. 

Let $$L_i^*=L_i-V(S_1) \text{ \quad for each  $i\in [2,s]$. }$$ 
Our goal now is to extend $f_1$ to a graceful labeling of $S$ by iteratively adding   $L_i^*$  back to $S_1$ 
applying Theorem~\ref{thm:attching lemma}.  Let 
\begin{numcases}{S_i=}
S_{i-1}(x)\oplus L^*_{i}(y_i) & \text{if $i\notin \{k_1, \ldots , k_r\}$}, \nonumber  \\ 
S_{i-1}(y_i)\oplus L^*_i(z_i) & \text{otherwise}.  \nonumber 
\end{numcases}
Let $i\in [2,s]$ and suppose that we have found a graceful labeling $f_{i-1}$ of  $S_{i-1}$ such that 
\begin{equation}\label{eqn:label-upper-bound-S1}
f_{i-1}(w) \le f_1(w)+\sum^{i-1}_{j=2} \frac{\ell_j}{2} \quad \text{for any $w\in V(S_1)$}. 
\end{equation}
Note that~\eqref{eqn:label-upper-bound-S1} holds for $f_1$, and that 
 for any $j\in [2,i-1]$, we have 
\begin{align}\label{eqn:upper-bound-li}
	\ell_j \le \frac{\ell_{j+1}-2}{2} \le  \ldots \le \frac{1}{2^{i-j}}\ell_i-2\sum_{h=1}^{i-j} \frac{1}{2^h} =\frac{1}{2^{i-j}}\ell_i-2(1-\frac{1}{2^{i-j}}).
\end{align} 
 We consider two cases regarding whether $i\in \{k_1, \ldots, k_r\}$. 

{\bf \noindent Case 1: $i\not\in \{k_1, \ldots, k_r\}$.}

  In order to apply Theorem~\ref{thm:attching lemma}, it suffices to show that 
  $f_{i-1}(x)+ \lfloor\frac{\ell_i}{2} \rfloor +1 \le \ell_i$.  
  As $f_1(x)=0$, by~\eqref{eqn:label-upper-bound-S1}, we have 
  \begin{align*}
  & f_{i-1}(x)+ \lfloor\frac{\ell_i}{2} \rfloor +1  \\
 & \le   \sum^{i-1}_{j=2} \frac{\ell_j}{2}+ \lfloor\frac{\ell_i}{2} \rfloor +1 \\ 
 & \overset{\eqref{eqn:upper-bound-li}}{\le}  \sum^{i-1}_{j=2} \frac{1}{2} \left( \frac{1}{2^{i-j}}\ell_i-2(1-\frac{1}{2^{i-j}})\right)+\lfloor\frac{\ell_i}{2} \rfloor +1 \\
  &\le \ell_i  \sum_{j=1}^{i-1}  \frac{1}{2^j} - \sum^{i-1}_{j=2}(1-\frac{1}{2^{i-j}}) +1\\ 
  &= (1-\frac{1}{2^{i-1}}) \ell_i -(i-2)+ (1- \frac{1}{2^{i-2}}) +1. 
   \end{align*}
  When   $i\ge 4$, it is clear that $f_{i-1}(x)+ \lfloor\frac{\ell_i}{2} \rfloor +1   \le \ell_i$. 
  When $i\in [2,3]$, since $\ell_i \ge 2\ell_{1}+2  \ge 4$,  we have $\frac{1}{2^{i-1}} \ell_i \ge 1$.
Thus  $f_{i-1}(x)+ \lfloor\frac{\ell_i}{2} \rfloor +1  \le \ell_i$. 
  
We let $f_i$ be the graceful labeling of $S_i$ given by Theorem~\ref{thm:attching lemma}. 
By the second part of the theorem, we have 
$f_i(w) =f_{i-1}(w)+  \lfloor \frac{\ell_i}{2}\rfloor $, and so $f_i$ satisfies~\eqref{eqn:label-upper-bound-S1}. 

 {\bf \noindent Case 2: $i\in \{k_1, \ldots, k_r\}$.}
 
Suppose that  $i=k_p$ for some $p\in [1,r]$. 
In this case, we have $f_1(y_i) =\ell_1+p$ by the definition of $f_1$. 
Note that $\ell_1\le \frac{1}{2}(\ell_2-4)$ and for each $j\in [2,i-1]$, we still have~\eqref{eqn:upper-bound-li}. 

To apply Theorem~\ref{thm:attching lemma},  we show below that 
$f_{i-1}(y_i)+ \lfloor\frac{\ell_i-1}{2} \rfloor +1 \le \ell_i-1$.  

As $f_1(y_i) =\ell_1+p$, by~\eqref{eqn:label-upper-bound-S1}, we have 
\begin{align*}
	& f_{i-1}(y_i)+ \lfloor\frac{\ell_i-1}{2} \rfloor +1  \\
&	\le   \ell_1+p+\sum^{i-1}_{j=2} \frac{\ell_j}{2}+ \lfloor\frac{\ell_i-1}{2} \rfloor +1 \\ 
&	\overset{\eqref{eqn:upper-bound-li}}{\le}  \frac{1}{2}(\ell_2-4)+p+ \sum^{i-1}_{j=2} \frac{1}{2} \left( \frac{1}{2^{i-j}}\ell_i-2(1-\frac{1}{2^{i-j}})\right)+\lfloor\frac{\ell_i-1}{2} \rfloor + 1\\
	&	\overset{\eqref{eqn:upper-bound-li}}{\le}   \frac{1}{2^{i-1}} \ell_i-(1-\frac{1}{2^{i-2}})-2+p+\ell_i  \sum_{j=1}^{i-1}  \frac{1}{2^j} - \sum^{i-1}_{j=2}(1-\frac{1}{2^{i-j}}) +\frac{1}{2}\\ 
		&= \frac{1}{2^{i-1}} \ell_i-(1-\frac{1}{2^{i-2}})-2+p+(1-\frac{1}{2^{i-1}}) \ell_i -(i-2)+ (1- \frac{1}{2^{i-2}}) +\frac{1}{2}\\
& 	=	\ell_i -2+p-(i-2)+ \frac{1}{2} \\
	&= \ell_i-(i-p)+\frac{1}{2}. 
\end{align*}
Since $i=k_p$ and $p\le i-1$, we get   $f_{i-1}(y_i)+ \lfloor\frac{\ell_i-1}{2} \rfloor +1  \le \ell_i- \frac{1}{2}$, 
which implies 
 $f_{i-1}(y_i)+ \lfloor\frac{\ell_i-1}{2} \rfloor +1   \le \ell_i- 1$ as   $f_{i-1}(y_i)+ \lfloor\frac{\ell_i}{2} \rfloor +1 $ is 
 an integer. 
 
 We let $f_i$ be the graceful labeling of $S_i$ given by Theorem~\ref{thm:attching lemma}. 
 By the second part of the theorem, we again have 
 $f_i(w)  =f_{i-1}(w)+  \lfloor \frac{\ell_i-1}{2}\rfloor $, and so $f_i$ satisfies~\eqref{eqn:label-upper-bound-S1}. 
 
 Continue the process above, we can find a graceful labeling $f_s$ for the spider $S$. 
\end{proof}

\section{Proof of Theorems~\ref{thm:graceful-construction} and~\ref{main3}}

In this section, we prove Theorems~\ref{thm:graceful-construction} and~\ref{main3}. Since the proof of Theorem~\ref{main3} is shorter when Theorem~\ref{thm:graceful-construction} is assumed, we begin with the latter.

\begin{proof}[Proof of Theorem~\ref{main3}]
	 Let $S$ be a spider 
	 with  all but three legs of length at most $2$.  If $S$ has at most one leg of length greater than two,
	 then we are done by Theorem~\ref{thm:graceful-construction}. 
	 Thus, we assume that $S$ has at least two legs of length greater than two.  Let the center of $S$
	 be $x$ and $L_1$ and $L_2$ be two legs of $S$ that each has length at least three. 
	 Let $S^*=S-(V(L_1\cup L_2) \setminus \{x\})$. 
	 By Theorem~\ref{thm:graceful-construction}, $S^*$ has a graceful labeling $g$ such that $g(x)=0$. 
	 By Lemma~\ref{lem:alpha-labeling-Pn}(b), $L_1\cup L_2$ has an $\alpha$-labeling $h$ with   $h(x)=0$. 
	 Now, by Lemma~\ref{lem:vertex amalgamation}, $S$ has a graceful labeling.  
\end{proof}

\begin{proof}[Proof of Theorem~\ref{thm:graceful-construction}]

It suffices to construct the desired graceful labeling for spiders in which all but at most one leg has length~$2$. Indeed, suppose a spider $S'$ has a graceful labeling $f'$ with center $x$ labeled $0$. Let $S^*$ be the spider obtained from $S'$ by adding $t$ new vertices $y_1,\ldots,y_t$ and $t$ new edges $xy_1,\ldots,xy_t$. Extend $f'$ to a labeling $f^*$ by setting $f^*(u)=f'(u)$ for every $u\in V(S')$ and
\[
f^*(y_i)=|E(S')|+i \qquad \text{for } i\in [1,t].
\]
Then $f^*$ is clearly graceful, and $f^*(x)=0$. Therefore, for the remainder of the proof, let $S$ be a spider with center $x_0$ such that all but at most one leg have length~$2$. We construct a graceful labeling $f$ of $S$ satisfying $f(x_0)=0$. 
Since a spider  with exactly two legs is  a path, by Lemma~\ref{lem:alpha-labeling-Pn}(a), we suppose that $S$ has at least three legs.

We let $N$ be the  leg of $S$ that has length larger than 2 if it exists or let it be an arbitrary  leg of $S$, 
 and let $L_1, \ldots, L_s$ be all the other  legs of $S$, which are all  of length 2,  where $s\ge 2$ is an integer.   Let 
\begin{align*}
	L_i&=x_0 u_iv_i  \quad  i\in [1,s],  \\
	N&= x_0x_1\ldots x_\ell  \quad \text{where $\ell =|E(N)|$}, \quad \text{and}\\
	m&= 2s+\ell. 
\end{align*}
Note that $m=|E(S)|$.  See Figure~\ref{Fig2} for  an illustration of the vertex names of $S$ when  $s=2$ and $\ell=8$. 

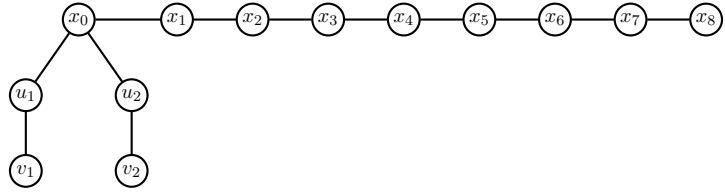
\begin{figure}[htbp]
	\begin{center}
		\begin{tikzpicture}[scale=1,thick]
			
		{\tikzstyle{every node}=[draw ,circle,fill=white, minimum size=0.6cm,
		inner sep=0pt]
			
			\node[circle, draw, fill=white, scale=.7] (x1) at (-8,-1) {$x_1$};
			\node[circle, draw, fill=white, scale=.7] (x2) at (-7,-1) {$x_2$};
			\node[circle, draw, fill=white, scale=.7] (x3) at (-6,-1) {$x_3$};
			\node[circle, draw, fill=white, scale=.7] (x4) at (-5,-1) {$x_4$};
			\node[circle, draw, fill=white, scale=.7] (x5) at (-4,-1) {$x_5$};
			\node[circle, draw, fill=white, scale=.7] (x6) at (-3,-1) {$x_6$};
			\node[circle, draw, fill=white, scale=.7] (x7) at (-2,-1) {$x_7$};
			\node[circle, draw, fill=white, scale=.7] (x8) at (-1,-1) {$x_8$};
			
			\node[circle, draw, fill=white, scale=.7] (x0) at (-9.3,-1) {$x_0$};
			\node[circle, draw, fill=white, scale=.7] (u1) at (-10,-2) {$u_1$};	
			\node[circle, draw, fill=white, scale=.7] (u2) at (-8.6,-2) {$u_2$};
			\node[circle, draw, fill=white, scale=.7] (v1) at (-10,-3) {$v_{1}$};		
			\node[circle, draw, fill=white, scale=.7] (v2) at (-8.6,-3) {$v_{2}$};
		
		}

			\path[draw,thick]
		(x0) edge node[name=la,pos=0.7, above] {\color{blue} } (u1)
		(x0) edge node[name=la,pos=0.7, above] {\color{blue} } (u2)
		(u1) edge node[name=la,pos=0.6,above] {\color{blue}  } (v1)
		(u2) edge node[name=la,pos=0.6,above] {\color{blue}  } (v2)
		;
		\draw (x0) -- (x1) -- (x2) -- (x3) --(x4)--(x5)--(x6)--(x7)--(x8);
		\end{tikzpicture}
	\end{center}
	\caption{An illustration of  naming the vertices of $S$.}
	\label{Fig2}
\end{figure}

We define a labeling $f: V(S) \to [0,m]$ in two cases according to the parity of $\ell$. 

{\bf \noindent Suppose $\ell$ is odd. }  

For any $i\in [1,s]$, let 
\begin{align*}
	f(u_{i})&=m-(2i-1),\\
	f(v_{i})&=2i-1; 
\end{align*}
and for any $i\in [0,\ell]$, let 
 \begin{align*}
	f(x_i)&=\begin{cases}
			2\times \frac{i}{4}  =\frac{i}{2} & \text{if $i\equiv 0 \pmod 4$},\\
		  m- 2\times \frac{i-1}{4}=m-\frac{i-1}{2}&  \text{if $i\equiv 1 \pmod 4$},\\
		 2s-1 + 2\times  \frac{i+2}{4} =2s-1 + \frac{i+2}{2}& \text{if $i\equiv 2 \pmod 4$}, \\
		m-(2s-1) -2\times \frac{i+1}{4} =m-(2s-1) -\frac{i+1}{2}& \text{if $i\equiv 3 \pmod 4$}.
		\end{cases}
\end{align*}
See Figure~\ref{Fig3} for an illustration of $f$ when $s=2$ and $\ell=11$. 

\begin{figure}[htbp]
	\begin{center}
		\begin{tikzpicture}[scale=1,thick]
			
			{\tikzstyle{every node}=[draw ,circle,fill=white, minimum size=0.7cm,
				inner sep=0pt]
				
				\node[circle, draw, fill=white, scale=.7] (x1) at (-8,-1) {$x_1$};
				\node[circle, draw, fill=white, scale=.7] (x2) at (-7,-1) {$x_2$};
				\node[circle, draw, fill=white, scale=.7] (x3) at (-6,-1) {$x_3$};
				\node[circle, draw, fill=white, scale=.7] (x4) at (-5,-1) {$x_4$};
				\node[circle, draw, fill=white, scale=.7] (x5) at (-4,-1) {$x_5$};
				\node[circle, draw, fill=white, scale=.7] (x6) at (-3,-1) {$x_6$};
				\node[circle, draw, fill=white, scale=.7] (x7) at (-2,-1) {$x_7$};
				\node[circle, draw, fill=white, scale=.7] (x8) at (-1,-1) {$x_8$};
					\node[circle, draw, fill=white, scale=.7] (x9) at (0,-1) {$x_9$};
					\node[circle, draw, fill=white, scale=.7] (x10) at (1,-1) {$x_{10}$};
					\node[circle, draw, fill=white, scale=.7] (x11) at (2,-1) {$x_{11}$};
				
				\node[circle, draw, fill=white, scale=.7] (x0) at (-9.3,-1) {$x_0$};
				\node[circle, draw, fill=white, scale=.7] (u1) at (-10,-2) {$u_1$};	
				\node[circle, draw, fill=white, scale=.7] (u2) at (-8.6,-2) {$u_2$};
				\node[circle, draw, fill=white, scale=.7] (v1) at (-10,-3) {$v_{1}$};		
				\node[circle, draw, fill=white, scale=.7] (v2) at (-8.6,-3) {$v_{2}$};
				
			}

				\node at  (-9.3,-0.5) {$0$};
				\node at	($(x1)+(0,0.5)$)  {15}; 
				\node at	($(x2)+(0,0.5)$)  {5}; 
				\node at	($(x3)+(0,0.5)$)  {10}; 
				\node at	($(x4)+(0,0.5)$)  {2}; 
				\node at	($(x5)+(0,0.5)$)  {13}; 
				\node at	($(x6)+(0,0.5)$)  {7}; 
				\node at	($(x7)+(0,0.5)$)  {8}; 
				\node at	($(x8)+(0,0.5)$)  {4}; 
				\node at	($(x9)+(0,0.5)$)  {11}; 
				\node at	($(x10)+(0,0.5)$)  {9}; 
					\node at	($(x11)+(0,0.5)$)  {6}; 
					
						\node at	($(u1)-(0.5,0)$)  {14}; 
							\node at	($(v1)-(0.5,0)$)  {1}; 
							
								\node at	($(u2)+(0.5,0)$)  {12}; 
									\node at	($(v2)+(0.5,0)$)  {3}; 
			
			\path[draw,thick]
			(x0) edge node[name=la,pos=0.7, above] {\color{blue} } (u1)
			(x0) edge node[name=la,pos=0.7, above] {\color{blue} } (u2)
			(u1) edge node[name=la,pos=0.6,above] {\color{blue}  } (v1)
			(u2) edge node[name=la,pos=0.6,above] {\color{blue}  } (v2)
			;
			\draw (x0) -- (x1) -- (x2) -- (x3) --(x4)--(x5)--(x6)--(x7)--(x8)--(x9)--(x10)--(x11);
		\end{tikzpicture}
	\end{center}
	\caption{A graceful labeling of $S$ when $s=2$ and $\ell=11$.}
	\label{Fig3}
\end{figure}
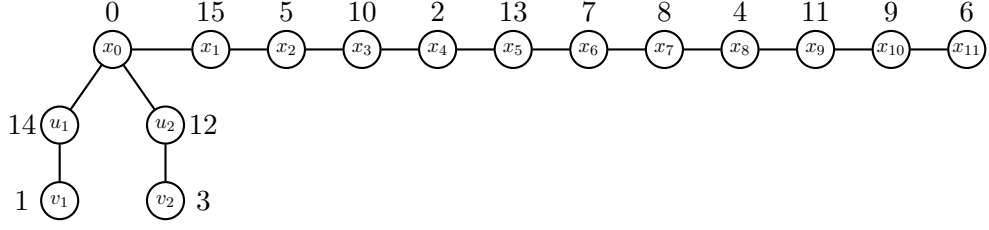

{\bf \noindent Suppose $\ell$ is even.} 

For any $i\in [1,s]$, let 
\begin{align*}
	f(u_{i})&=m-2(i-1),\\
	f(v_{i})&=2i-1; 
\end{align*}
and for any $i\in [0,\ell]$, let 
\begin{align*}
	f(x_i)&=\begin{cases}
		2\times \frac{i}{4}=\frac{i}{2} & \text{if $i\equiv 0 \pmod 4$}, \\ 
	m-2s -2\times \frac{i-1}{4}  =	m-2s -\frac{i-1}{2} &  \text{if $i\equiv 1 \pmod 4$},\\
	2s-1 + 2\times \frac{i+2}{4} =	2s-1 + \frac{i+2}{2}& \text{if $i\equiv 2 \pmod 4$}, \\
	m-1-2\times \frac{i-3}{4} =	m-1-\frac{i-3}{2} & \text{if $i\equiv 3 \pmod 4$}.
	\end{cases}
\end{align*}
See Figure~\ref{Fig4} for an illustration of $f$ when $s=2$ and $\ell=10$.  

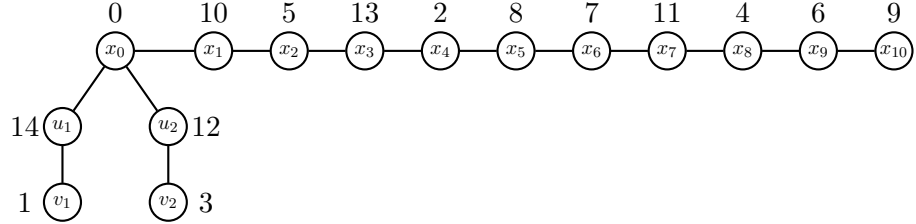
\begin{figure}[htbp]
	\begin{center}
		\begin{tikzpicture}[scale=1,thick]
			
			{\tikzstyle{every node}=[draw ,circle,fill=white, minimum size=0.7cm,
				inner sep=0pt]
				
				\node[circle, draw, fill=white, scale=.7] (x1) at (-8,-1) {$x_1$};
				\node[circle, draw, fill=white, scale=.7] (x2) at (-7,-1) {$x_2$};
				\node[circle, draw, fill=white, scale=.7] (x3) at (-6,-1) {$x_3$};
				\node[circle, draw, fill=white, scale=.7] (x4) at (-5,-1) {$x_4$};
				\node[circle, draw, fill=white, scale=.7] (x5) at (-4,-1) {$x_5$};
				\node[circle, draw, fill=white, scale=.7] (x6) at (-3,-1) {$x_6$};
				\node[circle, draw, fill=white, scale=.7] (x7) at (-2,-1) {$x_7$};
				\node[circle, draw, fill=white, scale=.7] (x8) at (-1,-1) {$x_8$};
				\node[circle, draw, fill=white, scale=.7] (x9) at (0,-1) {$x_9$};
				\node[circle, draw, fill=white, scale=.7] (x10) at (1,-1) {$x_{10}$};
				
				\node[circle, draw, fill=white, scale=.7] (x0) at (-9.3,-1) {$x_0$};
				\node[circle, draw, fill=white, scale=.7] (u1) at (-10,-2) {$u_1$};	
				\node[circle, draw, fill=white, scale=.7] (u2) at (-8.6,-2) {$u_2$};
				\node[circle, draw, fill=white, scale=.7] (v1) at (-10,-3) {$v_{1}$};		
				\node[circle, draw, fill=white, scale=.7] (v2) at (-8.6,-3) {$v_{2}$};
				
			}

			\node at  (-9.3,-0.5) {$0$};
			\node at	($(x1)+(0,0.5)$)  {10}; 
			\node at	($(x2)+(0,0.5)$)  {5}; 
			\node at	($(x3)+(0,0.5)$)  {13}; 
			\node at	($(x4)+(0,0.5)$)  {2}; 
			\node at	($(x5)+(0,0.5)$)  {8}; 
			\node at	($(x6)+(0,0.5)$)  {7}; 
			\node at	($(x7)+(0,0.5)$)  {11}; 
			\node at	($(x8)+(0,0.5)$)  {4}; 
			\node at	($(x9)+(0,0.5)$)  {6}; 
			\node at	($(x10)+(0,0.5)$)  {9};

			\node at	($(u1)-(0.5,0)$)  {14}; 
			\node at	($(v1)-(0.5,0)$)  {1}; 
			
			\node at	($(u2)+(0.5,0)$)  {12}; 
			\node at	($(v2)+(0.5,0)$)  {3}; 
			
			\path[draw,thick]
			(x0) edge node[name=la,pos=0.7, above] {\color{blue} } (u1)
			(x0) edge node[name=la,pos=0.7, above] {\color{blue} } (u2)
			(u1) edge node[name=la,pos=0.6,above] {\color{blue}  } (v1)
			(u2) edge node[name=la,pos=0.6,above] {\color{blue}  } (v2)
			;
			\draw (x0) -- (x1) -- (x2) -- (x3) --(x4)--(x5)--(x6)--(x7)--(x8)--(x9)--(x10);
		\end{tikzpicture}
	\end{center}
	\caption{A graceful labeling of $S$ when $s=2$ and $\ell=10$.}
	\label{Fig4}
\end{figure}

In the remainder of the proof, we verify that $f$ is a graceful labeling of $S$. 
Before proceed to the proof, we define a partition of  $V(S)$ and $E(S)$ as follows. 
Let 
\begin{align*}
U &= \{u_i: i\in [1,s]\}, \\ 
V &= \{v_i: i\in [1,s]\}, \quad \text{and}\\ 
W_j&=\{x_i: i\in [1,\ell],\, i \equiv j \pmod 4\}, \quad j\in [0,3]; 
\end{align*} 
and
\begin{align*}
E_1 &= \{x_0 u_i: i\in [1,s]\}, \\
E_2 &= \{u_iv_i: i\in [1,s]\}, \\
F_j &=\{x_ix_{i+1}: i\equiv j \pmod 4\}, \quad j\in [0,3]. 
\end{align*}

We separate the proof into two cases. In both cases, we show that $f$ is injective from $V(S)$ to $[0,m]$ and $f(E(S))= [1,m]$

{\bf \noindent  Case 1: $\ell$ is odd}.

\begin{CLA}\label{claim:case1-1-to-1}
 The function 	$f$ is injective from $V(S)$ to $[0,m]$. 
\end{CLA}
 \pf 
   Since $m$ is odd in this case, 
by the definition of $f$, we know that for any $v\in V(S)$, 
\begin{numcases}{f(v)\equiv }
	0 \pmod 2 & \text{if $v\in U \cup W_0\cup W_3$}, \nonumber  \\
	1 \pmod 2 & \text{if $v\in V \cup W_1\cup W_2$}.   \nonumber 
\end{numcases}
Thus $f(u) \ne f(v)$ if $u\in   U \cup W_0\cup W_3$ and $v\in V \cup W_1\cup W_2$. 
Therefore, to show that $f(u) \ne f(v)$ 
for distinct $u, v \in V(S)$,  we consider $u,v \in U \cup W_0\cup W_3$ or $u,v\in V \cup W_1\cup W_2$.

Consider first that  $u,v \in U \cup W_0\cup W_3$. 
Since $f(u_i) >f(u_j)$ when $i<j$ and $u_i,u_j\in U$,   $f(x_i)<f(x_j)<f(u_s)$ when $i<j$ and $x_i,x_j\in W_0$, 
and $f(u_s)>f(x_i)>f(x_j)$ when $i<j$ and $x_i,x_j\in W_3$,   it is left to show 
that $f(u) \ne f(v)$ when $u\in W_0$ and $v\in W_3$.  Let $u=x_i$ and $v=x_j$ 
for $i\equiv 0 \pmod 4$ and  $j\equiv 3 \pmod 4$ such that $i,j\in [0,\ell]$.  
If $f(u)=f(v)$, then we get $\frac{i}{2}=m-(2s-1) -\frac{j+1}{2} $.  Since $m=2s+\ell$, this implies that $i+j=2\ell +1$, a contradiction to $i,j\in [0,\ell]$.

Consider then  that  $u,v \in V \cup W_1\cup W_2$. 
Since $f(v_i) <f(v_j)$ when $i<j$ and $v_i,v_j\in V$,   $f(x_i)>f(x_j)>f(v_s)$ when $i<j$ and $x_i,x_j\in W_1$, 
and $f(v_s)<f(x_i)<f(x_j)$ when $i<j$ and $x_i,x_j\in W_2$,   it is left to show 
that $f(u) \ne f(v)$ when $u\in W_1$ and $v\in W_2$.  Let $u=x_i$ and $v=x_j$ 
for $i\equiv 1 \pmod 4$ and  $j\equiv 2 \pmod 4$ such that $i,j\in [1,\ell]$.  
If $f(u)=f(v)$, then we get $m-\frac{i-1}{2}=(2s-1) +\frac{j+2}{2} $.  This implies that $m=2s-1+\frac{i+j-1}{2}$, or equivalently $i+j=2(\ell+1)+1$, a contradiction to $i,j\in [1,\ell]$.  
 \qed 
 
 \begin{CLA}\label{claim:case1-edge-labels}
 	We have   $f(E(S))= [1,m]$. 
 \end{CLA}
 
 \pf 
Since $f$ is injective and so  $f(e)\in [1,m]$ for any $e\in E(S)$, it suffices to show that $f(e_1) \ne f(e_2)$ 
for distinct $e_1, e_2 \in E(S)$.

Since $m$ is odd, by the definition of $f$,
we know that for any $e\in E(S)$, 
\begin{numcases}{f(e)\equiv }
0 \pmod 2 & \text{if $e\in E_1 \cup F_1\cup F_3$}, \nonumber  \\
1 \pmod 2 & \text{if $e\in E_2 \cup F_0\cup F_2$}.   \nonumber 
\end{numcases}
Thus $f(e_1) \ne f(e_2)$ if $e_1\in   E_1 \cup F_1\cup F_3$ and $e_2\in E_2 \cup F_0\cup F_2$. 
Therefore, to show that $f(e_1) \ne f(e_2)$ 
for distinct $e_1, e_2 \in E(S)$,  we consider two collections of six subcases each:

\begin{minipage}{0.5\textwidth}
 \begin{enumerate} 
	\item  [Case 1e.1] $e_1, e_2\in E_1$. 
	\item  [Case 1e.2] $e_1, e_2\in F_1$.
	\item  [Case 1e.3] $e_1, e_2\in F_3$.
	\item  [Case 1e.4]  $e_1\in E_1$ and $e_2 \in F_1$. 
	\item [Case 1e.5]  $e_1\in E_1$ and $e_2\in F_3$. 
	\item [Case 1e.6] $e_1\in F_1$ and $e_2\in F_3$. 
\end{enumerate}
\end{minipage}
\begin{minipage}{0.5\textwidth}
\begin{enumerate} 
	\item  [Case 1o.1] $e_1, e_2\in E_2$. 
	\item  [Case 1o.2] $e_1, e_2\in F_0$.
	\item  [Case 1o.3] $e_1, e_2\in F_2$.
	\item  [Case 1o.4]  $e_1\in E_2$ and $e_2 \in F_0$. 
	\item [Case 1o.5]  $e_1\in E_2$ and $e_2\in F_2$. 
	\item [Case 1o.6] $e_1\in F_0$ and $e_2\in F_2$. 
\end{enumerate}
\end{minipage}
 
 We list the expression of $f(e)$ for $e\in E(S)$:
 \begin{align*}
	f(e)&=\begin{cases}
	m-(2i-1)-0=m-2i+1 & \text{if $e=x_0u_i \in E_1$}, \\ 
	|m-(2i-1)-(2i-1)| =	|m-4i+2| &  \text{if $e=u_iv_i\in E_2$},\\
	m-\frac{i+1-1}{2}-\frac{i}{2}=m-i& \text{if $e=x_ix_{i+1} \in F_0$}, \\
	m-\frac{i-1}{2}-(2s-1+\frac{i+1+2}{2})=m-2s-i& \text{if $e=x_ix_{i+1}\in F_1$}, \\
    |m-(2s-1)-\frac{i+1+1}{2}-(2s-1+\frac{i+2}{2})|=|m-4s-i|& \text{if $e=x_ix_{i+1}\in F_2$}, \\
    m-(2s-1)-\frac{i+1}{2}-\frac{i+1}{2}=m-2s-i& \text{if $e=x_ix_{i+1} \in F_3$}. \\
	\end{cases}
\end{align*}
 
{\bf \noindent Case 1e.1:} $e_1, e_2\in E_1$.  

Let $e_1=x_0u_i$ and $e_2=x_0u_j$ for some distinct $i,j\in [1,s]$. 
Then $$f(e_1) =m-2i+1 \ne m-2j+1 =f(e_2).$$

{\bf \noindent Case 1e.2:} $e_1, e_2\in F_1$.   

Let $e_1=x_ix_{i+1}$ and $e_2=x_jx_{j+1}$ for some distinct $i,j \equiv 1 \pmod 4$ and $i,j\in [1,\ell-1]$. 
Then 
  $$f(e_1)  = m-2s-i\ne m-2s-j =f(e_2).$$
 
 {\bf \noindent Case 1e.3:} $e_1, e_2\in F_3$.   
 
 Let $e_1=x_ix_{i+1}$ and $e_2=x_jx_{j+1}$ for some distinct $i,j \equiv 3 \pmod 4$ and $i,j\in [1,\ell-1]$. 
 Then 
 $$f(e_1) = m-2s-i \ne m-2s-j=f(e_2).$$
 
 {\bf \noindent Case 1e.4:} $e_1\in E_1$ and $e_2\in F_1$.   
 
 Let $e_1=x_0u_i$ and $e_2=x_jx_{j+1}$ for some $i\in [1,s]$ and   $j \equiv 1 \pmod 4$ with  $j\in [1,\ell-1]$.  
  Since $m=2s+\ell$, we have 
 $$f(e_1) =m-2i+1 \ge m-2s+1 \ne m-2s- j =f(e_2).$$
 
 {\bf \noindent Case 1e.5:} $e_1\in E_1$ and $e_2\in F_3$.   
 
 Let $e_1=x_0u_i$ and $e_2=x_jx_{j+1}$ for some $i\in [1,s]$ and   $j \equiv 3 \pmod 4$ with  $j\in [1,\ell-1]$.  
 Since $m=2s+\ell$, we have 
 $$f(e_1) =m-2i+1 \ge m-2s+1 \ne m-2s- j =f(e_2).$$
 
 {\bf \noindent Case 1e.6:} $e_1\in F_1$ and $e_2\in F_3$.   
 
 Let $e_1=x_ix_{i+1}$ and $e_2=x_jx_{j+1}$ for some $i \equiv 1 \pmod 4$,   $j \equiv 3 \pmod 4$ and $i,j\in [1,\ell-1]$.  
 Then 
 $$f(e_1) =m-2s-i \ne  m-2s-j  =f(e_2).$$
 
 {\bf \noindent Case 1o.1:} $e_1, e_2\in E_2$. 
 
 Let $e_1=u_iv_i$ and $e_2=u_jv_j$ for some $i,j\in [1,s]$. 
 Suppose, without loss of generality, that $i<j$. 
 Then $f(e_1)=|m-4i+2|$ and $f(e_2) = |m-4j+2|$. 
 Suppose to the contrary that $f(e_1)=f(e_2)$. 
 Since $i\ne j$ and $i<j$, we must have  $4j-2-m= m-4i+2$. 
 This implies that $2m =2(2i+2j-2)$ or $m=2i+2j-2$, a contradiction to the assumption that $m$ is odd.

 {\bf \noindent Case 1o.2:} $e_1, e_2\in F_0$.   
 
 Let $e_1=x_ix_{i+1}$ and $e_2=x_jx_{j+1}$ for some distinct  $i,j \equiv 0\pmod 4$ and $i,j\in [0,\ell-1]$. 
 Then 
 $$f(e_1) = m-i \ne m-j =f(e_2).$$

 {\bf \noindent Case 1o.3:} $e_1, e_2\in F_2$.   
 
Let $e_1=x_ix_{i+1}$ and $e_2=x_jx_{j+1}$ for some  $i,j \equiv 2\pmod 4$ and $i,j\in [1,\ell-1]$. 
Suppose, without loss of generality, that $i<j$. 
Then we have 
\begin{align*}
f(e_1) =|m-4s-i| \quad  \text{and}\quad & f(e_2) =|m-4s-j|. 
\end{align*}
 Since $i<j$, $f(e_1) =f(e_2)$ implies that $4s+j-m=m-4s-i$,
 and so $m=4s+\frac{i+j}{2}$. 
 Since $i,j \equiv 2\pmod 4$, $\frac{i+j}{2}$ is even and so $m$ is even.  
 However, this gives a contradiction to the assumption that $m$ is odd. 

 {\bf \noindent Case 1o.4:} $e_1\in E_2$ and  $e_2\in F_0$.   
 
 Let $e_1=u_iv_i$   and $e_2=x_jx_{j+1}$ for some $i\in [1,s]$ and   $j \equiv 0 \pmod 4$ with  $j\in [0,\ell-1]$.  
 Then $f(e_1)=|m-4i+2|$ and $f(e_2) =m-j $.  If $f(e_1) =f(e_2)$, 
 then we either have $j=4i-2$ or $m=2i-1+\frac{j}{2}$. 
 In the former case, we get a contradiction to  the fact that $j \equiv 0 \pmod 4$; 
 and in the latter case, we get a contradiction to $m=2s+\ell$. 
 
  {\bf \noindent Case 1o.5:} $e_1 \in E_2$ and  $e_2\in F_2$.   
  
   Let $e_1=u_iv_i$   and $e_2=x_jx_{j+1}$ for some $i\in [1,s]$ and   $j \equiv 2 \pmod 4$  with  $j\in [1,\ell-1]$.  
  Then $f(e_1)=|m-4i+2|$ and $f(e_2) =|m-4s-j|$.  If $f(e_1) =f(e_2)$, then we have either $m=2s+2i+\frac{j-2}{2}$ 
  or $4i=4s+j+2$. 
   In the former case, since $j \equiv 2 \pmod 4$ and so $\frac{j-2}{2}$ is even, we get a contradiction to  the fact that $m$ is odd; 
  and in the latter case, we get a contradiction to  the fact that $i\in [1,s]$.  
  
   {\bf \noindent Case 1o.6:} $e_1 \in F_0$ and  $e_2\in F_2$.   
   
   Let $e_1=x_ix_{i+1}$ and $e_2=x_jx_{j+1}$ for some  $i  \equiv 0\pmod 4$,  $j  \equiv 2\pmod 4$, and $i,j\in [0,\ell-1]$. 
   Then $f(e_1) = m-i$ and $f(e_2) =|m-4s-j|$. 
   If $f(e_1) =f(e_2)$, then we have either $m= 2s+\frac{i}{2} +\frac{j}{2}$ or  $i=4s+j$. 
   In the former case,  since $m=2s+\ell$, we get $i+j=2\ell$, a contradiction to $i,j\in [1,\ell-1]$. 
 In the latter case, we get a contradiction to  the fact that $i  \equiv 0\pmod 4$ and   $j  \equiv 2\pmod 4$. 
  \qed 
  
  Combining Claims~\ref{claim:case1-1-to-1} and~\ref{claim:case1-edge-labels}, we know that $f$ 
  is a graceful labeling of $S$ when $\ell$ is odd. 
  
  {\bf \noindent  Case 2: $\ell$ is even}.

  \begin{CLA} \label{claim:case2-1-to-1}
  	The function 	$f$ is injective from $V(S)$ to $[0,m]$. 
  \end{CLA}
  \pf 
  Since $m$ is even in this case, by the definition of $f$, we know that for any $v\in V(S)$,
  \begin{numcases}{f(v)\equiv }
	0 \pmod 2 & \text{if $v\in U \cup W_0\cup W_1$}, \nonumber  \\
	1 \pmod 2 & \text{if $v\in V \cup W_2\cup W_3$}.   \nonumber 
\end{numcases}

  Thus $f(u) \ne f(v)$ if $u\in   U \cup W_0\cup W_1$ and $v\in V \cup W_2\cup W_3$. 
Therefore, to show that $f(u) \ne f(v)$ 
for distinct $u, v \in V(S)$,  we consider $u,v \in U \cup W_0\cup W_1$ or $u,v\in V \cup W_2\cup W_3$. 

Consider first that  $u,v \in U \cup W_0\cup W_1$. 
Since $f(u_i) >f(u_j)$ when $i<j$ and $u_i,u_j\in U$,   $f(x_i)<f(x_j)<f(u_s)$ when $i<j$ and $x_i,x_j\in W_0$, 
and $f(u_s)>f(x_i)>f(x_j)$ when $i<j$ and $x_i,x_j\in W_1$,  it is left to show 
that $f(u) \ne f(v)$ when $u\in W_0$ and $v\in W_1$.  Let $u=x_i$ and $v=x_j$ 
for $i\equiv 0 \pmod 4$ and  $j\equiv 1 \pmod 4$ such that $i,j\in [0,\ell]$.  
If $f(u)=f(v)$, then we get $\frac{i}{2}=m-2s-\frac{j-1}{2} $.  Since $m=2s+\ell$, this implies that $i+j=2\ell +1$, a contradiction to $i,j\in [0,\ell]$. 

Consider then  that  $u,v \in V \cup W_2\cup W_3$. 
Since $f(v_i) <f(v_j)$ when $i<j$ and $v_i,v_j\in V$,   $f(x_j)>f(x_i)>f(v_s)$ when $i<j$ and $x_i,x_j\in W_2$, 
and $f(v_s)<f(x_j)<f(x_i)$ when $i<j$ and $x_i,x_j\in W_3$,   it is left to show 
that $f(u) \ne f(v)$ when $u\in W_2$ and $v\in W_3$.  Let $u=x_i$ and $v=x_j$ 
for $i\equiv 2 \pmod 4$ and  $j\equiv 3 \pmod 4$ such that $i,j\in [1,\ell]$.  
If $f(u)=f(v)$, then we get $2s-1+\frac{i+2}{2}=m-1-\frac{j-3}{2} $.  This implies that $i+j=2\ell+1$, a contradiction to $i,j\in [1,\ell]$.  
  \qed

   \begin{CLA} \label{claim:case2-edge-labels}
  	We have $f(E(S))= [1,m]$. 
  \end{CLA}
  
  \pf 
  Since $f$ is injective and so  $f(e)\in [1,m]$ for any $e\in E(S)$, it suffices to show that $f(e_1) \ne f(e_2)$ 
for distinct $e_1, e_2 \in E(S)$. 

Since $m$ is even, by the definition of $f$,
we know that for any $e\in E(S)$, 
\begin{numcases}{f(e)\equiv }
0 \pmod 2 & \text{if $e\in E_1 \cup F_0\cup F_2$}, \nonumber  \\
1 \pmod 2 & \text{if $e\in E_2 \cup F_1\cup F_3$}.   \nonumber 
\end{numcases}
Thus $f(e_1) \ne f(e_2)$ if $e_1\in   F_1 \cup E_0\cup E_2$ and $e_2\in F_2 \cup E_1\cup E_3$. 
Therefore, to show that $f(e_1) \ne f(e_2)$ 
for distinct $e_1, e_2 \in E(S)$,  we consider two collections of six subcases each:

\begin{minipage}{0.5\textwidth}
 \begin{enumerate} 
	\item  [Case 2e.1] $e_1, e_2\in E_1$. 
	\item  [Case 2e.2] $e_1, e_2\in F_0$.
	\item  [Case 2e.3] $e_1, e_2\in F_2$.
	\item  [Case 2e.4]  $e_1\in E_1$ and $e_2 \in F_0$. 
	\item [Case 2e.5]  $e_1\in E_1$ and $e_2\in F_2$. 
	\item [Case 2e.6] $e_1\in F_0$ and $e_2\in F_2$. 
\end{enumerate}
\end{minipage}
\begin{minipage}{0.5\textwidth}
\begin{enumerate} 
	\item  [Case 2o.1] $e_1, e_2\in E_2$. 
	\item  [Case 2o.2] $e_1, e_2\in F_1$.
	\item  [Case 2o.3] $e_1, e_2\in F_3$.
	\item  [Case 2o.4]  $e_1\in E_2$ and $e_2 \in F_1$. 
	\item [Case 2o.5]  $e_1\in E_2$ and $e_2\in F_3$. 
	\item [Case 2o.6] $e_1\in F_1$ and $e_2\in F_3$. 
\end{enumerate}
\end{minipage}

  We list the expression of $f(e)$ for $e\in E(S)$:
 \begin{align*}
	f(e)&=\begin{cases}
		m-2(i-1)-0=m-2i+2 & \text{if $e=x_0u_i \in E_1$}, \\ 
	|m-2i+2-(2i-1)| =|m-4i+3| &  \text{if $e=u_iv_i \in E_2$},\\
	m-2s-\frac{i+1-1}{2}-\frac{i}{2}=m-2s-i& \text{if $e=x_ix_{i+1} \in F_0$}, \\
	|m-2s-\frac{i-1}{2}-(2s-1+\frac{i+1+2}{2})|=|m-4s-i|& \text{if $e=x_ix_{i+1} \in F_1$}, \\
    m-1-\frac{i+1-3}{2}-(2s-1+\frac{i+2}{2})=m-2s-i& \text{if $e=x_ix_{i+1} \in F_2$}, \\
    m-1-\frac{i-3}{2}-\frac{i+1}{2}=m-i& \text{if $e=x_ix_{i+1} \in F_3$}. \\
	\end{cases}
\end{align*}
{\bf \noindent Case 2e.1:} $e_1, e_2\in E_1$.  

Let $e_1=x_0u_i$ and $e_2=x_0u_j$ for some distinct $i,j\in [1,s]$. 
Then $$f(e_1) =m-2i+2 \ne m-2j+2 =f(e_2).$$

{\bf \noindent Case 2e.2:} $e_1, e_2\in F_0$.   

Let $e_1=x_ix_{i+1}$ and $e_2=x_jx_{j+1}$ for some distinct  $i,j \equiv 0 \pmod 4$ and $i,j\in [0,\ell]$. 
Then 
  \begin{align*}
f(e_1) =m-2s-i\ne m-2s-j=f(e_2)
\end{align*}
 
 {\bf \noindent Case 2e.3:} $e_1, e_2\in F_2$.   
 
 Let $e_1=x_ix_{i+1}$ and $e_2=x_jx_{j+1}$ for some distinct  $i,j \equiv 2 \pmod 4$ and $i,j\in [2,\ell]$. 
 Then 
 $$f(e_1)= m-2s-i \ne m-2s-j=f(e_2).$$
 
 {\bf \noindent Case 2e.4:} $e_1\in E_1$ and $e_2\in F_0$.   
 
 Let $e_1=x_0u_i$ and $e_2=x_jx_{j+1}$ for some $i\in [1,s]$ and   $j \equiv 0 \pmod 4$ with  $j\in [0,\ell]$.  
  Since $m=2s+\ell$, we have 
 $$f(e_1) =m-2i+2 \ge m-2s+2 \ne m-2s- j=f(e_2).$$
 
 {\bf \noindent Case 2e.5:} $e_1\in E_1$ and $e_2\in F_2$.   
 
 Let $e_1=x_0u_i$ and $e_2=x_jx_{j+1}$ for some $i\in [1,s]$ and   $j \equiv 2\pmod 4$ with  $j\in [2,\ell]$.  
 Since $m=2s+\ell$, we have 
 $$f(e_1) =m-2i+2 \ge m-2s+2 \ne m-2s- j =f(e_2).$$
 
 {\bf \noindent Case 2e.6:} $e_1\in F_0$ and $e_2\in F_2$.   
 
 Let $e_1=x_ix_{i+1}$ and $e_2=x_jx_{j+1}$ for some $i \equiv 0 \pmod 4$,   $j \equiv 2 \pmod 4$ and $i,j\in [0,\ell]$.  
Then 
 $$f(e_1) =m-2s-i \ne  m-2s-j  =f(e_2).$$
 
 {\bf \noindent Case 2o.1:} $e_1, e_2\in E_2$. 
 
 Let $e_1=u_iv_i$ and $e_2=u_jv_j$ for some $i,j\in [1,s]$. 
 Suppose, without loss of generality, that $i<j$. 
 Then $f(e_1)=|m-4i+3|$ and $f(e_2) = |m-4j+3|$. 
 Suppose to the contrary that $f(e_1)=f(e_2)$. 
 Since $i\ne j$ and $i<j$, we must have  $4j-3-m= m-4i+3$. 
 This implies that $2m =4i+4j-6$ or $m=2i+2j-3$, a contradiction to the assumption that $m$ is even.

 {\bf \noindent Case 2o.2:} $e_1, e_2\in F_1$.   
 
Let $e_1=x_ix_{i+1}$ and $e_2=x_jx_{j+1}$ for some  $i,j \equiv 1 \pmod 4$ and $i,j\in [1,\ell]$. 
Suppose, without loss of generality, that $i<j$. 
Since $m=2s+\ell$, we have 
  \begin{align*}
f(e_1) =|m-4s-i| \quad  \text{and}\quad & f(e_2) =|m-4s-j|. 
\end{align*}
 Since $i<j$, $f(e_1) =f(e_2)$ implies that $4s+j-m=m-4s-i$,
 and so $m=4s+\frac{i+j}{2}$. 
 Since $i,j \equiv 1\pmod 4$, $\frac{i+j}{2}$ is odd and so $m$ is odd.  
 However, this gives a contradiction to the assumption that $m$ is even. 
 
 {\bf \noindent Case 2o.3:} $e_1, e_2\in F_3$.   
 
Let $e_1=x_ix_{i+1}$ and $e_2=x_jx_{j+1}$ for some distinct  $i,j \equiv 3\pmod 4$ and $i,j\in [3,\ell]$. 
Then 
\begin{align*}
f(e_1) =m-i\ne m-j=f(e_2) 
\end{align*}

 {\bf \noindent Case 2o.4:} $e_1\in E_2$ and  $e_2\in F_1$.   
 
 Let $e_1=u_iv_i$   and $e_2=x_jx_{j+1}$ for some $i\in [1,s]$ and   $j \equiv 1 \pmod 4$ with  $j\in [1,\ell]$.  
 Then $f(e_1)=|m-4i+3|$ and $f(e_2) =|m-4s-j| $.  If $f(e_1) =f(e_2)$, 
 then we either have $j=4(i-s)-3$ or $m=2s+2i+\frac{j-3}{2}$. In the former case, since $i\le s$, we get  a contradiction to $j\in [1,\ell]$.
 In the latter case, since $j \equiv 1 \pmod 4$ and so $2s+2i+\frac{j-3}{2}$ is odd, we get a contradiction to  the fact that $m$ is even. 
 
  {\bf \noindent Case 2o.5:} $e_1 \in E_2$ and  $e_2\in F_3$.   
  
   Let $e_1=u_iv_i$   and $e_2=x_jx_{j+1}$ for some $i\in [1,s]$ and   $j \equiv 3 \pmod 4$  with  $j\in [3,\ell]$.  
  Then $f(e_1)=|m-4i+3|$ and $f(e_2) =m-j$.  If $f(e_1) =f(e_2)$, then we have either $j=4i-3$ 
  or $m=2i+\frac{j-3}{2}$. In the former case, we get a contradiction to  the fact that $j\equiv3\pmod 4$; 
  and in the latter case,  since $2i+\frac{j-3}{2} <2s+\ell$, we get a contradiction to the fact  $m=2s+\ell$. 

   {\bf \noindent Case 2o.6:} $e_1 \in F_1$ and  $e_2\in F_3$.   
   
   Let $e_1=x_ix_{i+1}$ and $e_2=x_jx_{j+1}$ for some  $i  \equiv 1\pmod 4$,  $j  \equiv 3\pmod 4$, and $i,j\in [1,\ell]$. 
   Then $f(e_1) =|m-4s-i|$ and $f(e_2) =m-j$. 
   If $f(e_1) =f(e_2)$, then we have either $j=4s+i$ or  $2m=4s+i+j$. 
   In the former case,  we get a contradiction to the fact that $i  \equiv 1\pmod 4$ and $j\equiv 3\pmod 4$. 
 In the latter case, since $m=2s+\ell$, we get $i+j=2\ell$, a contradiction to $i,j\in[1,\ell -1]$. 

Combining Claims~\ref{claim:case2-1-to-1} and~\ref{claim:case2-edge-labels}, we know that $f$ 
  is a graceful labeling of $S$ when $\ell$ is even. \qed

The proof of Theorem~\ref{thm:graceful-construction} is now completed. 
\end{proof}

\bibliographystyle{abbrv}
\bibliography{reference}

\end{document}